\documentclass{birkau}

\usepackage{amsmath,amssymb,url}
\usepackage{hyperref}
\usepackage[all,cmtip]{xy}
\usepackage{enumitem}
\usepackage{mathtools}
\input diagxy

\theoremstyle{plain}
 \newtheorem{thm}{Theorem}[section]
 \newtheorem{lem}[thm]{Lemma}
 \newtheorem{prop}[thm]{Proposition}
 \newtheorem{cor}[thm]{Corollary}
\theoremstyle{definition}
 
 \newtheorem{exmp}[thm]{Example}
 \newtheorem{rem}[thm]{Remark}

\DeclareMathAlphabet{\mathcal}{OMS}{cmsy}{m}{n}

\DeclareMathOperator{\ev}{ev}

\def\oto{{\bfig\morphism<180,0>[\mkern-4mu`\mkern-4mu;]\place(86,0)[\circ]\efig}}

\newcommand{\ra}{\rightarrow}

\newcommand{\lda}{\swarrow}
\newcommand{\rda}{\searrow}

\newcommand{\bv}{\bigvee}

\newcommand{\dv}{\dashv}

\newcommand{\nat}{\natural}

\renewcommand{\phi}{\varphi}
\newcommand{\al}{\alpha}
\newcommand{\be}{\beta}

\newcommand{\ga}{\gamma}

\newcommand{\lam}{\lambda}

\newcommand{\sC}{\mathsf{C}}
\newcommand{\sD}{\mathsf{D}}
\newcommand{\sE}{\mathsf{E}}

\newcommand{\sQ}{\mathsf{Q}}

\newcommand{\Fy}{\mathfrak{y}}

\newcommand{\Cat}{\mathbf{Cat}}

\newcommand{\Set}{\mathbf{Set}}
\newcommand{\CcSet}{\mathbf{CcSet}}

\newcommand{\QSet}{\sQ\text{-}\Set}

\newcommand{\RSet}{[0,1]_*\text{-}\Set}

\newcommand{\RSetM}{[0,1]_*\text{-}\underline{\Set}}
\newcommand{\RCcSetM}{[0,1]_*\text{-}\underline{\CcSet}}
\newcommand{\RDist}{[0,1]_*\text{-}\mathbf{Dist}}

\newcommand{\sCd}{\sC^{\dag}}

\newcommand{\Fyd}{\Fy^{\dag}}

\newcommand{\CdX}{\sCd X}

\renewcommand{\leq}{\leqslant}

\numberwithin{equation}{section}
\allowdisplaybreaks

\begin{document}

\title{Cartesian closedness of the category of real-valued sets, II}

\author[L. Shen]{Lili Shen}
\address{School of Mathematics, Sichuan University, Chengdu 610064, China}
\email{shenlili@scu.edu.cn}

\corrauthor[J. Zhang]{Jian Zhang}
\address{School of Mathematics, Sichuan University, Chengdu 610064, China}
\email{zhangjian.2025@qq.com}



\begin{abstract}
Let $[0,1]_*$ be the unit interval $[0,1]$ equipped with a left-continuous t-norm $*$. We prove that $[0,1]_*\text{-}\mathbf{Set}$ is cartesian closed if and only if $*$ is the minimum t-norm. This extends the classification for continuous t-norms established in the first part of this work to the whole left-continuous setting.
\end{abstract}

\maketitle

\section{Introduction}

Building on the theory of frame-valued sets initiated by Higgs \cite{Higgs1970,Higgs1984} and Fourman--Scott \cite{Fourman1979}, H\"ohle and his collaborators \cite{Hoehle1992,Hoehle1995a,Hoehle1995b,Hoehle2011a,Pu2012} developed the theory of quantale-valued sets, which associates with a unital involutive quantale $\sQ$ a category 
\[
\QSet
\]
of sets valued in $\sQ$. From the perspective of enriched category theory \cite{Stubbe2014}, these are symmetric categories enriched in the quantaloid $\sD\sQ$; their morphisms are left adjoint $\sD\sQ$-distributors. When $\sQ$ is a frame, $\QSet$ is a topos \cite{Fourman1979}. For a general quantale, however, $\QSet$ need not be a topos \cite{Hu2025}, and it can even fail to be cartesian closed \cite{Shen2026}.

In this paper, the table $[0,1]_*$ of truth-values is the unit interval equipped with a left-continuous t-norm. The first part of this work \cite{Shen2026} established that, when $*$ is continuous, $\RSet$ is cartesian closed if and only if $*$ is the minimum t-norm.

The left-continuous case cannot be treated as a formal extension of the continuous one. Indeed, Lai and Luo \cite{Lai2026a} show that, after continuity is dropped, the category $[0,1]_*\text{-}\Cat$ of categories enriched in the unital quantale $([0,1],*,1)$ can be cartesian closed for certain non-minimum left-continuous t-norms; see also the general quantaloid characterization of Stubbe and Yu \cite{Stubbe2026}. The category $[0,1]_*\text{-}\Cat$ is a global-enrichment setting, whereas $\RSet$ allows varying type data. This contrast is not explained merely by the availability of variable types: $\RSet$ has symmetric $\sD[0,1]_*$-categories with variable types as objects and left adjoint $\sD[0,1]_*$-distributors as morphisms, whereas $[0,1]_*\text{-}\Cat$ has $[0,1]_*$-categories and $[0,1]_*$-functors. We prove that $\RSet$ is cartesian closed only for the minimum t-norm; in particular, the non-minimum cases in the classification of Lai and Luo do not yield exceptions here.

Our proof differs substantially from the one in Part~I \cite{Shen2026}. Rather than using the structural description of continuous t-norms to select a local configuration and then deriving a full formula for the homs of a hypothetical exponential, we attach a canonical local obstruction to each non-idempotent $c\in[0,1]$. Namely, we consider the idempotent floor
\[
a_c=\bigvee\{e\leq c\mid e*e=e\}.
\]
The preservation of suprema by $*$ makes $a_c$ idempotent, and hence $a_c<c$. From the pair $a_c<c$ we construct two small variable-type test objects. If a suitable exponential existed, their Cauchy completions would force two lower bounds on homs in the exponential. Evaluation gives an incompatible upper bound. Thus the argument applies throughout the left-continuous setting without a full calculation of the homs of a hypothetical exponential.

The paper is organized as follows. Section \ref{real-valued-sets} recalls real-valued sets and their distributor presentation. In Section \ref{cartesian-closedness-of-real-valued-sets}, we first collect the required facts about Cauchy completion, then prove the canonical obstruction and the cartesian-closedness characterization, and finally compare it with the result of Lai and Luo for $[0,1]_*\text{-}\Cat$.

\section{Real-valued sets}
\label{real-valued-sets}

Throughout this paper, we consider $[0,1]_*$ as the table of truth-values, where $*$ is a left-continuous t-norm on the unit interval $[0,1]$. More generally, a binary operation $*$ on an interval $[a,b]$ is a \emph{left-continuous t-norm} \cite{Klement2000,Klement2004b,Alsina2006} if
\begin{itemize}
\item $([a,b],*,b)$ is a commutative monoid,
\item $p*\bigvee\limits_{i\in I}q_{i}=\bigvee\limits_{i\in I}(p*q_{i})$ and $(\bigvee\limits_{i\in I}p_{i})*q=\bigvee\limits_{i\in I}(p_{i}*q)$ for all $p,q,p_i,q_i\in[a,b]$ $(i\in I)$.
\end{itemize}
We write $[a,b]_*$ for the resulting t-norm structure.

Note that for each $p\in[a,b]$, there is a Galois connection
\[(p*-)\dv(p\ra -)\colon[a,b]\to[a,b]\]
satisfying
\[p*q\leq r\iff q\leq p\ra r\]
for all $p,q,r\in[a,b]$. 

Let $p\in[a,b]$. We say that $p$ is \emph{idempotent} if $p*p=p$.


\begin{exmp} \label{MLP-def}
The \emph{nilpotent minimum t-norm} $*_{\mathrm{NM}}$ on $[0,1]$ is a left-continuous but non-continuous t-norm, which is given by
\[
p*_{\mathrm{NM}}q=
\begin{cases}
0 & \text{if}\ p+q\leq1,\\
p\wedge q & \text{if}\ p+q>1.
\end{cases}
\]
Its residual is
\[
p\ra q=
\begin{cases}
1 & \text{if}\ p\leq q,\\
(1-p)\vee q & \text{if}\ p>q.
\end{cases}
\]
\end{exmp}

In fact,
\[([a,b],*,b)\]
is a commutative unital \emph{quantale} \cite{Rosenthal1990}. 

Given $q,u\in[a,b]$, we say that $u$ is \emph{divisible by $q$} if
\[u=q*(q\ra u).\]

\begin{prop} \label{divisible} 
In every left-continuous t-norm $[a,b]_*$,
\begin{equation}\label{divisible:u-v-q}
v*(q\ra u)=(q\ra v)*u
\end{equation}
whenever $u,v\in[a,b]$ are divisible by $q$.
\end{prop}

\begin{proof}
Since $u=q*(q\ra u)$ and $v=q*(q\ra v)$, associativity and commutativity give
\[
v*(q\ra u)=q*(q\ra v)*(q\ra u)=(q\ra v)*u.\qedhere
\]
\end{proof}

Following the definition of \emph{quantale-valued set} \cite{Hoehle1992,Hoehle1995b,Hoehle2011a,Pu2012,Lai2020}, by a \emph{$[0,1]_*$-set} (or a \emph{$[0,1]_*$-valued set}) we mean a set $X$ equipped with a map
\[\al\colon X\times X\to[0,1],\]
such that
\begin{enumerate}[label=(S\arabic*)]
\item \label{S1} $\al(x,y)=\al(x,x)*(\al(x,x)\ra\al(x,y))=\al(y,y)*(\al(y,y)\ra\al(x,y))$,
\item \label{S2} $\al(x,y)=\al(y,x)$,
\item \label{S3} $\al(y,z)*(\al(y,y)\ra\al(x,y))\leq\al(x,z)$
\end{enumerate}
for all $x,y,z\in X$. Here 
\[\al(x,y)\] 
is understood as the truth-value of the statement that \emph{$x$ is equal to $y$}. 

In particular, it necessarily holds that
\begin{equation}\label{al(x,y) leq al(x,x) wedge al(y,y)}
\al(x,y)\leq\al(x,x)\wedge\al(y,y).
\end{equation}
For $x\in X$, we call the value $\al(x,x)$ the \emph{type}, or \emph{extent}, of $x$, which may be understood as the \emph{extent of existence} of $x$.

Let $(X,\al)$ be a $[0,1]_*$-set. For $x,y\in X$, we write $x\cong y$ if
\begin{equation} \label{x-cong-y-def}
\al(x,x)=\al(y,y)=\al(x,y),
\end{equation}
and we say that $(X,\al)$ is \emph{separated} if 
\[x\cong y\iff x=y.\]
Moreover, we denote by
\begin{equation} \label{X_q-def}
X_q=\{x\in X\mid\al(x,x)=q\}
\end{equation}
the slice of $X$ consisting of elements of type $q$.

A \emph{distributor} 
\[\phi\colon (X,\al)\oto (Y,\be)\]
of $[0,1]_*$-sets is a function
\[\phi\colon X\times Y\to[0,1]\]
such that
\begin{enumerate}[label=(M\arabic*)]
\item \label{M1} $\phi(x,y)=\al(x,x)*(\al(x,x)\ra\phi(x,y))=\be(y,y)*(\be(y,y)\ra\phi(x,y))$,
\item \label{M2} $(\be(y,y)\ra\be(y,y'))*\phi(x,y)*(\al(x,x)\ra\al(x',x))\leq\phi(x',y')$
\end{enumerate}
for all $x,x'\in X$, $y,y'\in Y$; it necessarily holds that
\begin{equation}\label{phi(x,y) leq al(x,x) wedge be(y,y)}
\phi(x,y)\leq\al(x,x)\wedge\be(y,y).
\end{equation}
The \emph{opposite} of $\phi$ is the distributor
\[\phi^{\circ}\colon(Y,\be)\oto(X,\al),\quad \phi^{\circ}(y,x)=\phi(x,y).\]
Its composite
\[\psi\circ\phi\colon(X,\al)\oto(Z,\ga)\]
with another distributor $\psi\colon(Y,\be)\oto(Z,\ga)$ is given by
\begin{equation} \label{psi-circ-phi-def}
(\psi\circ\phi)(x,z)=\bv_{y\in Y}\psi(y,z)*(\be(y,y)\ra\phi(x,y)),
\end{equation}
with
\begin{equation} \label{al-id}
\al\colon(X,\al)\oto(X,\al)
\end{equation}
serving as the identity distributor for this composition. The category of $[0,1]_*$-sets and their distributors is denoted by
\[\RDist.\]
$\RDist$ is a \emph{quantaloid} \cite{Rosenthal1996,Stubbe2005}; that is, each hom-set is a complete lattice, the order on distributors is pointwise, and the composition preserves suprema in each variable, i.e.,
\[\psi\circ\Big(\bigvee_{i\in I} \phi_i\Big)=\bigvee_{i\in I}\psi\circ \phi_i\quad\text{and}\quad\Big(\bigvee_{i\in I}\psi_i\Big)\circ \phi=\bigvee_{i\in I}\psi_i\circ \phi\]
for all distributors $\phi,\phi_i\colon (X,\al)\oto(Y,\be)$, $\psi,\psi_i\colon(Y,\be)\oto (Z,\ga)$ $(i\in I)$. The corresponding right adjoints induced by the composition maps
\[(-\circ \phi)\dashv(-\lda \phi)\colon\RDist(X,Z)\to\RDist(Y,Z)\]
and
\[(\psi\circ -)\dashv(\psi\rda -)\colon\RDist(X,Z)\to\RDist(X,Y)\]
satisfy
\[\psi\circ \phi\leq\xi\iff \psi\leq\xi\lda \phi\iff \phi\leq \psi\rda\xi\]
for any $\phi\colon (X,\al)\oto (Y,\be)$, $\psi\colon (Y,\be)\oto (Z,\gamma)$, $\xi\colon (X,\al)\oto (Z,\gamma)$. 

A \emph{morphism} 
\[\phi\colon(X,\al)\oto(Y,\be)\] 
of $[0,1]_*$-sets is a left adjoint distributor, whose right adjoint is necessarily its opposite (see \cite[Proposition 3.5.3]{Heymans2010} and \cite[Theorem 3.7]{Heymans2011}); equivalently, it is a distributor satisfying
\begin{enumerate}[label=(M\arabic*)]
\setcounter{enumi}{2}
\item \label{M3} $\phi\circ\phi^{\circ}\leq\be$,
\item \label{M4} $\al\leq\phi^{\circ}\circ\phi$.
\end{enumerate}
The category of $[0,1]_*$-sets and their morphisms is denoted by
\[\RSet.\]

We may also consider a \emph{monotone function}
\[f\colon(X,\al)\to(Y,\be),\]
which is a map $f\colon X\to Y$ such that
\begin{equation} \label{monotone-function-def}
\al(x,x)=\be(fx,fx)\quad\text{and}\quad\al(x,x')\leq\be(fx,fx')
\end{equation}
for all $x,x'\in X$. The category of $[0,1]_*$-sets and monotone functions is denoted by
\[\RSetM.\]

\begin{rem}
Every left-continuous t-norm $[0,1]_*$ gives rise to a quantaloid $\sD[0,1]_*$ \cite{Hoehle2011,Pu2012,Stubbe2014,Lai2020}. Its objects are the elements of $[0,1]$, and, for $p,q\in[0,1]$,
\[\sD[0,1]_*(p,q)=\{u\in[0,1]\mid u\text{ is divisible by both }p\text{ and }q\}.\]
For $u\in\sD[0,1]_*(p,q)$ and $v\in\sD[0,1]_*(q,r)$, the composite is $v\circ u=v*(q\ra u)$; the identity $\sD[0,1]_*$-arrow of $q$ is $q$ itself, and each hom-set is ordered as a subset of $[0,1]$.

The $[0,1]_*$-sets and distributors defined above are precisely symmetric $\sD[0,1]_*$-categories and $\sD[0,1]_*$-distributors, respectively. Accordingly, a morphism is precisely a left adjoint $\sD[0,1]_*$-distributor, and a monotone function is precisely a $\sD[0,1]_*$-functor. Thus, $\RSet$ is the category of symmetric $\sD[0,1]_*$-categories and left adjoint $\sD[0,1]_*$-distributors, while $\RSetM$ is the category of symmetric $\sD[0,1]_*$-categories and $\sD[0,1]_*$-functors.
\end{rem}

Every monotone function $f\colon(X,\al)\to(Y,\be)$ induces a morphism 
\begin{equation} \label{f-graph-def}
f_{\nat}\colon(X,\al)\oto(Y,\be),\quad f_{\nat}(x,y)=\be(fx,y)
\end{equation}
of $[0,1]_*$-sets, called the \emph{graph} of $f$, whose opposite
\[f^{\nat}\colon(Y,\be)\oto(X,\al),\quad f^{\nat}(y,x)=\be(y,fx),\]
is called the \emph{cograph} of $f$. Obviously, for every $[0,1]_*$-set $(X,\al)$, the identity function $1_X$ is monotone, and
\[\al=(1_X)_{\nat}=1_X^{\nat}.\]
Hence, in order to simplify the notation, we abbreviate a $[0,1]_*$-set $(X,\al)$ to $X$, and write $1_X^{\nat}(x,y)$ instead of $\al(x,y)$ if no confusion arises. 

For each $q\in[0,1]$, we have a one-element $[0,1]_*$-set $\{q\}$ with 
\[1_{\{q\}}^{\nat}(q,q)=q.\]

For each distributor $\phi\colon X\oto Y$,
\[\phi(x,-)\colon\{1_X^{\nat}(x,x)\}\oto Y\]
and
\[\phi(-,y)\colon X\oto \{1_Y^{\nat}(y,y)\}\]
are distributors for all $x\in X$, $y\in Y$. In particular, $\phi(x,y)$ can be considered as a distributor from $\{1_X^{\nat}(x,x)\}$ to $\{1_Y^{\nat}(y,y)\}$, so \ref{M2} can be written as
\begin{enumerate}[label=(M\arabic*$^\prime$)]
\setcounter{enumi}{1}
\item \label{M2'} $1_Y^{\nat}(y,y')\circ\phi(x,y)\circ 1_X^{\nat}(x',x)\leq\phi(x',y')$,
\end{enumerate}
\ref{S3} can be written as
\begin{enumerate}[label=(S\arabic*$^\prime$)]
\setcounter{enumi}{2}
\item \label{S3'} $1_X^{\nat}(y,z)\circ 1_X^{\nat}(x,y)\leq 1_X^{\nat}(x,z)$,
\end{enumerate}
and the composite distributor \eqref{psi-circ-phi-def} can be written as
\begin{equation} \label{psi-circ-phi-def'}
(\psi\circ\phi)(x,z)=\bv_{y\in Y}\psi(y,z)\circ\phi(x,y).
\end{equation}

\begin{lem} \label{phi-p-q}
For $p,q\in[0,1]$, the following statements are equivalent:
\begin{enumerate}[label={\rm(\roman*)}]
\item \label{phi-p-q:phi} There exists a morphism $\phi\colon\{p\}\oto\{q\}$ of one-element $[0,1]_*$-sets.
\item \label{phi-p-q:p} $p$ is divisible by $q$ and satisfies $p=p*(q\ra p)$.
\end{enumerate}
In this case, it necessarily holds that $\phi(p,q)=p$. Therefore, a morphism between one-element $[0,1]_*$-sets, when it exists, will be denoted by
\[\overline{p}\colon\{p\}\oto\{q\}.\]
When the codomain needs to be indicated, we write $\overline{p}_q$ instead of $\overline{p}$.
\end{lem}
\begin{proof}
\ref{phi-p-q:phi}$\implies$\ref{phi-p-q:p}: Suppose that $\phi(p,q)=r$. Then, by \ref{M1} and \ref{M4} we have
\begin{equation} \label{phi-p-q:M134}
r\leq p\wedge q\quad\text{and}\quad p\leq r*(q\ra r).
\end{equation}
Thus
\[p\leq r*(q\ra r)\leq r\leq p,\] 
and consequently $r=p$. By \ref{M1},
\[p=r=q*(q\ra r)=q*(q\ra p),\]
so $p$ is divisible by $q$. Moreover,
\[p\leq r*(q\ra r)=p*(q\ra p)\leq p;\]
hence
\begin{equation} \label{p=p*(q-ra-p)}
p=p*(q\ra p).
\end{equation}
\ref{phi-p-q:p}$\implies$\ref{phi-p-q:phi}: In this case, it is easy to see that $\phi(p,q)=p$ satisfies \ref{M1}--\ref{M4}:
\begin{itemize}
\item \ref{M1} $p=p*(p\ra p)=q*(q\ra p)$.
\item \ref{M2} $(q\ra q)*p*(p\ra p)\leq p$.
\item \ref{M3} $p*(p\ra p)\leq q$.
\item \ref{M4} $p\leq p*(q\ra p)$.
\end{itemize}
Therefore, $\phi\colon\{p\}\oto\{q\}$ is a morphism of one-element $[0,1]_*$-sets.
\end{proof} 

\section{Cartesian closedness of the category of real-valued sets}
\label{cartesian-closedness-of-real-valued-sets}

A \emph{singleton} on a $[0,1]_*$-set $X$ is a morphism $\lam\colon\{p\}\oto X$ whose domain is a one-element $[0,1]_*$-set. A $[0,1]_*$-set is \emph{Cauchy complete} precisely when every singleton is represented by an element.

\begin{prop} \label{Cauchy-complete-def} (See \cite[Proposition 7.1]{Stubbe2005}.)
A $[0,1]_*$-set $X$ is Cauchy complete if and only if for each singleton $\lam\colon\{p\}\oto X$, there exists $x\in X$ such that
\[\lam=1_X^{\nat}(x,-).\]
In this case, necessarily $1_X^{\nat}(x,x)=p$.
\end{prop}

For every $[0,1]_*$-set $X$, let
\[\CdX\coloneqq\bigcup_{q\in[0,1]}\RSet(\{q\},X)\]
be the set of all singletons on $X$. Its hom is
\begin{equation} \label{CdX-hom}
1_{\CdX}^{\nat}(\lam,\lam')={\lam'}^{\circ}\circ\lam.
\end{equation}
The $[0,1]_*$-set $\CdX$ is separated and Cauchy complete \cite[Proposition 7.12]{Stubbe2005}; it is called the \emph{Cauchy completion} of $X$. For a singleton $\lam\colon\{q\}\oto X$, one has
\begin{equation} \label{CdX-type}
1_{\CdX}^{\nat}(\lam,\lam)=q.
\end{equation}
The canonical monotone function
\begin{equation} \label{FYd-def}
\Fyd_X\colon X\to\CdX,\quad \Fyd_X x=1_X^{\nat}(x,-)
\end{equation}
is the unit of the Cauchy-completion reflection.

The following two elementary descriptions are obtained by the same formal calculations as in \cite[Examples 3.3 and 3.4]{Shen2026}. The difference is that Lemma \ref{phi-p-q} replaces the continuous-case description of morphisms between one-element $[0,1]_*$-sets:

\begin{exmp} \label{Cdq-exmp}
By Lemma \ref{phi-p-q}, the Cauchy completion of the one-element $[0,1]_*$-set $\{q\}$ is
\begin{equation} \label{Cdq}
\sCd\{q\}=\{\overline{p}\mid p\text{ is divisible by }q\text{ and }p=p*(q\ra p)\}.
\end{equation}
For $\overline{p},\overline{p'}\in\sCd\{q\}$,
\begin{equation} \label{Cdp-hom}
1_{\sCd\{q\}}^{\nat}(\overline{p},\overline{p'})=p'*(q\ra p).
\end{equation}
\end{exmp}

\begin{exmp} \label{Cdq-x-y-exmp}
Let $X=\{x,y\}$ be a two-element $[0,1]_*$-set. Every singleton on $X$ has a presentation
\[
1_X^{\nat}(z,-)\circ\overline{p}\colon\{p\}\oto X,
\]
where $z\in\{x,y\}$ and $\overline{p}\colon\{p\}\oto\{1_X^{\nat}(z,z)\}$ is a morphism. Hence
\[
\sCd X=\{1_X^{\nat}(z,-)\circ\overline{p}\mid z\in\{x,y\},\ \overline{p}\colon\{p\}\oto\{1_X^{\nat}(z,z)\}\}.
\]
\end{exmp}

Let $\RCcSetM$ be the full subcategory of $\RSetM$ consisting of separated Cauchy complete $[0,1]_*$-sets. It is reflective in $\RSetM$, with reflector $\sCd$ (see \cite[Proposition 7.14]{Stubbe2005}). The following standard equivalence can be proved by the same argument as \cite[Proposition 5.7(2)]{Pu2012}; the argument remains valid for left-continuous t-norms in the present setting (cf. \cite{Stubbe2005,Hoehle2011a}). It allows us to study the cartesian closedness of $\RSet$ in $\RCcSetM$.

\begin{prop} \label{RSet-RCcSetM}
The category $\RSet$ is equivalent to $\RCcSetM$.
\end{prop}

Recall that an object $Y$ of a category with finite products is \emph{exponentiable} if the functor $-\times Y$ has a right adjoint $(-)^Y$. Thus, if $Z^Y$ exists, there is an evaluation
\[\ev_Z\colon Z^Y\times Y\to Z\]
such that every monotone function $K\colon A\times Y\to Z$ has a unique transpose
\[\sE K\colon A\to Z^Y\]
with $\ev_Z\circ(\sE K\times1_Y)=K$.

The category $\RCcSetM$ has finite products. For $A,B\in\RCcSetM$,
\begin{equation} \label{X-times-Y-def}
A\times B=\{(x,y)\mid x\in A,\ y\in B,\ 1_A^{\nat}(x,x)=1_B^{\nat}(y,y)\}
\end{equation}
and
\[
1_{A\times B}^{\nat}\left((x,y),(x',y')\right)=1_A^{\nat}(x,x')\wedge1_B^{\nat}(y,y').
\]

We first isolate the only consequence of the exponential adjunction that will be used below.

\begin{lem} \label{Z^Y-hom}
Let $B$ be a $[0,1]_*$-set, let $Y,Z\in\RCcSetM$, and suppose that $Z^Y$ exists in $\RCcSetM$. For every monotone function
\[
K\colon\sCd B\times Y\to Z,
\]
put $k=\sE K\colon\sCd B\to Z^Y$. Then
\[
1_B^{\nat}(b,b')\leq1_{Z^Y}^{\nat}\bigl(k(\Fyd_Bb),k(\Fyd_Bb')\bigr)
\]
for all $b,b'\in B$.
\end{lem}

\begin{proof}
The composite $k\circ\Fyd_B\colon B\to Z^Y$ is monotone. The assertion follows immediately from \eqref{monotone-function-def}.
\end{proof}

\begin{lem} \label{sup-idem}
For $c\in[0,1]$, put
\[
a_c=\bigvee\{e\leq c\mid e*e=e\}.
\]
Then $a_c$ is idempotent. Consequently, if $c$ is non-idempotent, then $a_c<c$, and every idempotent element below $c$ is at most $a_c$.
\end{lem}

\begin{proof}
Let $E_c=\{e\leq c\mid e*e=e\}$. Since $*$ preserves suprema in each variable,
\[
a_c*a_c=\bigvee_{e,e'\in E_c}e*e'.
\]
For $e,e'\in E_c$, one has $e*e'\leq e\leq a_c$, whereas the diagonal terms give $e=e*e\leq a_c*a_c$. Hence $a_c*a_c=a_c$. If $a_c=c$, then $c$ is idempotent, so a non-idempotent $c$ satisfies $a_c<c$. The final assertion is immediate from the definition of $a_c$.
\end{proof}

\begin{prop} \label{RCcSetM-cc}
The category $\RCcSetM$ is cartesian closed if and only if $*$ is the minimum t-norm on $[0,1]$.
\end{prop}

\begin{proof}
If $*$ is the minimum t-norm, then $\RSet$ is the topos of $[0,1]$-valued sets over the frame $[0,1]$ \cite{Hoehle1992}. It is therefore cartesian closed, and so is the equivalent category $\RCcSetM$ by Proposition \ref{RSet-RCcSetM}.

Conversely, suppose that $*$ is not the minimum t-norm. There is then a non-idempotent $c\in[0,1]$: indeed, if every element were idempotent, then $p*q=p\wedge q$ for all $p,q\in[0,1]$. By Lemma \ref{sup-idem}, $a_c<c$.

Let $X=\{x,x'\}$ be the $[0,1]_*$-set with
\[
1_X^{\nat}(x,x)=1_X^{\nat}(x',x')=1,\qquad
1_X^{\nat}(x,x')=a_c.
\]
Put $Y=\sCd\{1\}$ and $Z=\sCd X$. An element of $Y$ is of the form $\overline{p}_1\colon\{p\}\oto\{1\}$, where $p$ is idempotent. If $(\overline{p}_c,\overline{p}_1)$ belongs to $\sCd\{c\}\times Y$, then $p$ is an idempotent element below $c$, and hence $p\leq a_c$. Since $a_c$ and $p$ are idempotent, this also gives $a_c*p=p$.

For idempotent $p,q$, direct use of \eqref{CdX-hom} gives
\[
1_Z^{\nat}\left(1_X^{\nat}(x,-)\circ\overline{p}_1,1_X^{\nat}(x,-)\circ\overline{q}_1\right)=p*q
\]
and
\[
1_Z^{\nat}\left(1_X^{\nat}(x,-)\circ\overline{p}_1,1_X^{\nat}(x',-)\circ\overline{q}_1\right)=a_c*p*q.
\]
Thus the following maps are monotone:
\[
F\colon\sCd\{1\}\times Y\to Z,\qquad
F(\overline{p}_1,\overline{p}_1)=1_X^{\nat}(x,-)\circ\overline{p}_1,
\]
\[
G\colon\sCd\{c\}\times Y\to Z,\qquad
G(\overline{p}_c,\overline{p}_1)=1_X^{\nat}(x,-)\circ\overline{p}_1,
\]
and
\[
H\colon\sCd\{1\}\times Y\to Z,\qquad
H(\overline{p}_1,\overline{p}_1)=1_X^{\nat}(x',-)\circ\overline{p}_1.
\]
For $G$, the preceding observation and the displayed hom calculations show monotonicity by the product formula.

Suppose that $Z^Y$ exists. Since $Z^Y$ is Cauchy complete, the exponential adjunction and the Cauchy-completion reflection give natural bijections (cf. \cite[(4.vi)--(4.viii)]{Shen2026})
\[
Z^Y_q\cong\RCcSetM(\sCd\{q\}\times Y,Z)
\]
for $q\in[0,1]$. Let $f\in Z^Y_1$, $g\in Z^Y_c$, and $h\in Z^Y_1$ be the elements corresponding to $F$, $G$, and $H$, respectively.

Let $B=\{b_0,b_1\}$ be the $[0,1]_*$-set with
\[
1_B^{\nat}(b_0,b_0)=1,\qquad
1_B^{\nat}(b_1,b_1)=c,\qquad
1_B^{\nat}(b_0,b_1)=c.
\]
Define a map $K\colon\sCd B\times Y\to Z$ by
\[
K\left(1_B^{\nat}(b_i,-)\circ\lam,\overline{p}_1\right)
=1_X^{\nat}(x,-)\circ\overline{p}_1,
\]
where $\lam\colon\{p\}\oto\{1_B^{\nat}(b_i,b_i)\}$ is a morphism. By Example \ref{Cdq-x-y-exmp}, this defines $K$ on all of $\sCd B\times Y$. If a point has two such presentations, the resulting value of $K$ is the same for both presentations. Moreover, the hom in the source product is at most $p*q$, which is exactly the hom between the corresponding images in $Z$. Hence $K$ is monotone.

For $i=0,1$, the singleton $1_B^{\nat}(b_i,-)$ induces a monotone function
\[
j_i\colon\sCd\{1_B^{\nat}(b_i,b_i)\}\to\sCd B,\qquad
j_i(\lam)=1_B^{\nat}(b_i,-)\circ\lam.
\]
The restrictions of $K$ along $j_0\times1_Y$ and $j_1\times1_Y$ are $F$ and $G$, respectively. Under the natural bijections above, naturality with respect to $j_0$ and $j_1$ identifies the corresponding values of the transpose $\sE K$ with $f$ and $g$; explicitly,
\[
(\sE K)(\Fyd_B b_0)=f,\qquad (\sE K)(\Fyd_B b_1)=g.
\]
Lemma \ref{Z^Y-hom} therefore gives
\[
c\leq1_{Z^Y}^{\nat}(f,g).
\]
The reverse inequality follows from the types of $f$ and $g$, so
\[
1_{Z^Y}^{\nat}(f,g)=c.
\]

Similarly, let $B'=\{b'_0,b'_1\}$ be the $[0,1]_*$-set with
\[
1_{B'}^{\nat}(b'_0,b'_0)=c,\qquad
1_{B'}^{\nat}(b'_1,b'_1)=1,\qquad
1_{B'}^{\nat}(b'_0,b'_1)=c.
\]
Define $K'\colon\sCd B'\times Y\to Z$ by
\[
K'\left(1_{B'}^{\nat}(b'_0,-)\circ\lam,\overline{p}_1\right)
=1_X^{\nat}(x,-)\circ\overline{p}_1
\]
and
\[
K'\left(1_{B'}^{\nat}(b'_1,-)\circ\lam,\overline{p}_1\right)
=1_X^{\nat}(x',-)\circ\overline{p}_1.
\]
If a point admits both presentations used in the definition of $K'$, then its type is an idempotent below $c$, hence is at most $a_c$; the two prescriptions for $K'$ therefore give the same element of $Z$, because $a_c*p=p$. In a mixed case, the parameter of a singleton presented through $b'_0$ is likewise at most $a_c$; consequently the mixed hom in $Z$ is $a_c*p*q=p*q$. Since the hom in the source product is at most $p*q$, the map $K'$ is monotone.

For $i=0,1$, define
\[
j'_i\colon\sCd\{1_{B'}^{\nat}(b'_i,b'_i)\}\to\sCd B',\qquad
j'_i(\lam)=1_{B'}^{\nat}(b'_i,-)\circ\lam.
\]
The restrictions of $K'$ along $j'_0\times1_Y$ and $j'_1\times1_Y$ are $G$ and $H$, respectively. The same naturality argument gives
\[
(\sE K')(\Fyd_{B'} b'_0)=g,\qquad (\sE K')(\Fyd_{B'} b'_1)=h.
\]
Hence Lemma \ref{Z^Y-hom} gives
\[
c\leq1_{Z^Y}^{\nat}(g,h).
\]
The reverse inequality follows from the types of $g$ and $h$, so
\[
1_{Z^Y}^{\nat}(g,h)=c.
\]

Let $\overline{1}_1$ denote the type-$1$ element of $Y$. The monotonicity of evaluation gives
\[
1_{Z^Y}^{\nat}(f,h)\leq1_Z^{\nat}\left(\ev_Z(f,\overline{1}_1),\ev_Z(h,\overline{1}_1)\right).
\]
The defining property of the transposes and the displayed calculation of the hom in $Z$ identify the right-hand side with
\[
1_Z^{\nat}\left(1_X^{\nat}(x,-),1_X^{\nat}(x',-)\right)=a_c.
\]
On the other hand, $1_{Z^Y}^{\nat}(g,g)=c$, and axiom \ref{S3} applied to $f,g,h$ yields
\[
c*(c\ra c)\leq1_{Z^Y}^{\nat}(f,h).
\]
This is impossible because $c*(c\ra c)=c>a_c$. Hence $Z^Y$ does not exist, and $\RCcSetM$ is not cartesian closed.
\end{proof}

The main result follows at once from Propositions \ref{RSet-RCcSetM} and \ref{RCcSetM-cc}:

\begin{thm} \label{QSet-topos-frame}
The category $\RSet$ is cartesian closed if and only if $*$ is the minimum t-norm on $[0,1]$.
\end{thm}

\begin{rem}

Theorem \ref{QSet-topos-frame} contrasts with the result of Lai and Luo \cite{Lai2026a}; see also the general quantaloid characterization of Stubbe and Yu \cite{Stubbe2026}. In $[0,1]_*\text{-}\Cat$, enrichment is global: all hom-values lie in the single quantale $[0,1]_*$, so there is no varying type data. Lai and Luo characterize the left-continuous t-norms for which $[0,1]_*\text{-}\Cat$ is cartesian closed; their classification includes t-norms different from the minimum t-norm.

By contrast, $\RSet$ has symmetric $\sD[0,1]_*$-categories with variable types as objects and left adjoint $\sD[0,1]_*$-distributors as morphisms. The two test objects $B$ and $B'$ in the proof of Proposition \ref{RCcSetM-cc} use the non-idempotent type $c$ alongside type $1$, a configuration unavailable in the global-enrichment setting of $[0,1]_*\text{-}\Cat$. Thus none of the non-minimum cases in the classification of Lai and Luo yields a cartesian closed $\RSet$. This contrast is not attributable to variable types alone: symmetry and the choice of left adjoint distributors are also part of the $\RSet$ framework.
\end{rem}

\begin{cor}
The category $\RSet$ is a topos if and only if $*$ is the minimum t-norm on $[0,1]$.
\end{cor}

\begin{proof}
If $*$ is the minimum t-norm, then $\RSet$ is a topos \cite{Hoehle1992}. Conversely, every topos is cartesian closed, so the conclusion follows from Theorem \ref{QSet-topos-frame}.
\end{proof}



\section*{Acknowledgements}

During the preparation of this manuscript, the authors used ChatGPT to assist with manuscript organization, language editing, mathematical consistency checking, and submission-readiness review. All AI-assisted suggestions were critically reviewed and, where appropriate, edited and verified by the authors, who remain fully responsible for the content of the manuscript.

\section*{Declarations}

\subsection*{Ethical approval}
Not applicable.

\subsection*{Competing interests}
The authors declare that they have no competing interests.

\subsection*{Authors' contributions}
All authors contributed equally.

\subsection*{Availability of data and materials}
Not applicable.

\subsection*{Funding}

The authors acknowledge support from the National Natural Science Foundation of China (Grant No. 12671561).

\end{document}